\documentclass[12pt,reqno]{amsart}
\usepackage{amsfonts}
\usepackage{amsmath}
\usepackage{amssymb,latexsym}
\usepackage{enumerate}
\usepackage[bookmarksnumbered, colorlinks, plainpages]{hyperref}
\usepackage{amsbsy}
\usepackage{amscd}
\usepackage{epsfig}
\usepackage{graphicx}
\usepackage{epstopdf}
\usepackage{caption}
\usepackage{subcaption}

\newtheorem{theorem}{Theorem}[section]

\newtheorem{proposition}{Proposition}[section]

\newtheorem{definition}{Definition}[section]

\newtheorem{example}{Example}[section]
\newtheorem{remark}{Remark}[section]

\numberwithin{equation}{section}

\begin{document}
\date{{\scriptsize Received: , Accepted: .}}
\title[NEW FIXED FIGURE RESULTS WITH THE NOTION OF $k$-ELLIPSE]{NEW FIXED
FIGURE RESULTS WITH THE NOTION OF $k$-ELLIPSE}
\author[N. TA\c{S}]{Nihal TA\c{S}}
\address{Bal\i kesir University, Department of Mathematics, 10145 Bal\i
kesir, Turkey}
\email{nihaltas@balikesir.edu.tr}
\author[H. AYT\.{I}MUR]{H\"{u}lya AYT\.{I}MUR}
\address{Bal\i kesir University, Department of Mathematics, 10145 Bal\i
kesir, Turkey}
\email{hulya.aytimur@balikesir.edu.tr}
\author[\c{S}. G\"{U}VEN\c{C}]{\c{S}aban G\"{U}VEN\c{C}}
\address{Bal\i kesir University, Department of Mathematics, 10145 Bal\i
kesir, Turkey}
\email{sguvenc@balikesir.edu.tr}
\maketitle

\begin{abstract}
In this paper, as a geometric approach to the fixed-point theory, we prove
new fixed-figure results using the notion of $k$-ellipse on a metric space.
For this purpose, we are inspired by the Caristi type contraction, Kannan
type contraction, Chatterjea type contraction and \'{C}iri\'{c} type
contraction. After that, we give some existence and uniqueness theorems of a
fixed $k$-ellipse. We also support our obtained results with illustrative
examples. Finally, we present a new application to the $S$-Shaped Rectified
Linear Activation Unit ($SReLU$) to show the importance of our theoretical
results.\newline
\end{abstract}

%% The title of the paper goes here.  Edit your title.

%% Now edit the following to give First Author name and address:
%% $^*$ for the corresponding author.

%% If there are three of more authors they are added in the obvious way.

%------------------------------------------------------------------------------------%
%%
%% Use the following command to make the title for the paper.
%
%\CoverPage

%
%%% The following environment is needed for the abstract.
%%%

\section{\textbf{Introduction and Background}}

\label{sec:intro} Recently, the geometry of the fixed-point set $%
Fix(T)=\left\{ x\in X:Tx=x\right\} $ of a self-mapping $T:X\rightarrow X$
has been studied as a new approach to the fixed-point theory. This approach
began with the \textquotedblleft fixed-circle problem\textquotedblright\ in 
\cite{Ozgur-malaysian}. This problem gains the importance since the
self-mapping has non-unique fixed points and the set of non-unique fixed
points includes some geometric figures such as circle, disc, ellipse etc.
For example, let us consider the usual metric space $\left( 
%TCIMACRO{\U{211d} }%
%BeginExpansion
\mathbb{R}
%EndExpansion
,d\right) $ and the self-mapping $T:%
%TCIMACRO{\U{211d} }%
%BeginExpansion
\mathbb{R}
%EndExpansion
\rightarrow 
%TCIMACRO{\U{211d} }%
%BeginExpansion
\mathbb{R}
%EndExpansion
$ defined as%
\begin{equation*}
Tx=\left\{ 
\begin{array}{ccc}
x & , & x\geq -1 \\ 
-x & , & x<-1%
\end{array}%
\right. \text{,}
\end{equation*}%
for all $x\in 
%TCIMACRO{\U{211d} }%
%BeginExpansion
\mathbb{R}
%EndExpansion
$. Then we have 
\begin{equation*}
C_{0,1}=\left\{ x\in 
%TCIMACRO{\U{211d} }%
%BeginExpansion
\mathbb{R}
%EndExpansion
:\left\vert x\right\vert =1\right\} =\left\{ -1,1\right\} \subset
Fix(T)=[-1,\infty )\text{,}
\end{equation*}%
that is, the fixed-point set $Fix(T)$ includes the unit circle. Similarly,
we get%
\begin{equation*}
D_{0,1}=\left\{ x\in 
%TCIMACRO{\U{211d} }%
%BeginExpansion
\mathbb{R}
%EndExpansion
:\left\vert x\right\vert \leq 1\right\} =\left[ -1,1\right] \subset
Fix(T)=[-1,\infty )\text{,}
\end{equation*}%
that is, the fixed-point set $Fix(T)$ includes the unit disc.

The notion of a fixed figure was defined as a generalization of the notions
of a fixed circle as follows:

A geometric figure $\mathcal{F}$ (a circle, an ellipse, a hyperbola etc.)
contained in the fixed point set $Fix\left( T\right) $ is called a\textit{\
fixed figure} (a fixed circle, a fixed ellipse, a fixed hyperbola etc.) of
the self-mapping $T$ (see \cite{Ozgur-figure}). For example, some
fixed-figure theorems were obtained using different techniques (see, \cite%
{Ercinar}, \cite{Joshi}, \cite{Ozgur-figure} and \cite{Tas-chapter} for more
details).

A $k$-ellipse is the locus of points of the plane whose sum of distances to
the $k$ foci is a constant $d$. The $1$-ellipse is the circle, and the $2$%
-ellipse is the classic ellipse. $k$-ellipses can be considered as
generalizations of ellipses. These special curves allow more than two foci 
\cite{NPS-2008}. $k$-ellipses have many names such as: $n$-ellipse \cite%
{Sekino}, multifocal ellipse \cite{Erdos}, polyellipse \cite{Melzak} and
egglipse \cite{Sahadevan}. It is noteworthy to mention that these curves
were first studied by Scottish mathematician and scientist James Clerk
Maxwell in 1846 \cite{Maxwell}.

From the above reasons, in this paper, we investigate new solutions to the
fixed-figure problem. Therefore, we use some known contractive conditions
such as Caristi type contraction, Kannan type contraction, Chatterjea type
contraction and \'{C}iri\'{c} type contraction to obtain some existence and
uniqueness theorems of a fixed $k$-ellipse. Also, we give necessary examples
to support our obtained results. Finally, we present a new application to
the $S$-Shaped Rectified Linear Activation Unit ($SReLU$) to show the
importance of our theoretical results.

\section{\textbf{Main Results}}

\label{sec:1} In this section, we give the definition of a $k$-ellipse with
some examples and prove new fixed-figure theorems on metric spaces. Also, we
investigate some properties of the obtained theorems with necessary examples.

\begin{definition}
\label{dfn1} Let $\left( X,d\right) $ be a metric space. The $k$-ellipse is
defined by%
\begin{equation*}
E\left[ x_{1},...,x_{r};r\right] =\left\{ x\in X:\sum_{i=1}^{k}d\left(
x,x_{i}\right) =r\right\} .
\end{equation*}%
Clearly, if $i=2$ then we get an ellipse and if $i=1$ then we get a circle
on a metric space with the same radius.
\end{definition}

Now, we give the following examples deal with $k$-ellipses.

\begin{example}
\label{sekil1} Let $\left( X=%
%TCIMACRO{\U{211d} }%
%BeginExpansion
\mathbb{R}
%EndExpansion
^{2},d\right) $ be a metric space with the metric $d:X\times X\rightarrow 
%TCIMACRO{\U{211d} }%
%BeginExpansion
\mathbb{R}
%EndExpansion
$ defined as%
\begin{equation*}
d\left( a,b\right) =\left\vert u_{1}-u_{2}\right\vert +\left\vert
v_{1}-v_{2}\right\vert \text{,}
\end{equation*}%
such that $a=\left( u_{1},v_{1}\right) ,b=\left( u_{2},v_{2}\right) \in X$.
Let us define $3$-ellipse for $x_{1}=\left( 1,0\right) ,$ $x_{2}=\left(
0,0\right) ,$ $x_{3}=\left( 0,1\right) $ as follows $($see Figure \ref{fig:1}%
$):$%
\begin{equation*}
E\left[ x_{1},x_{2},x_{3};r\right] =\left\{ p\left( x,y\right) \in
X:\left\vert x-1\right\vert +\left\vert y\right\vert +\left\vert
x\right\vert +\left\vert y\right\vert +\left\vert x\right\vert +\left\vert
y-1\right\vert =r\right\} .
\end{equation*}
\end{example}

\begin{figure}[h]
\centering
\includegraphics[width=.4\linewidth]{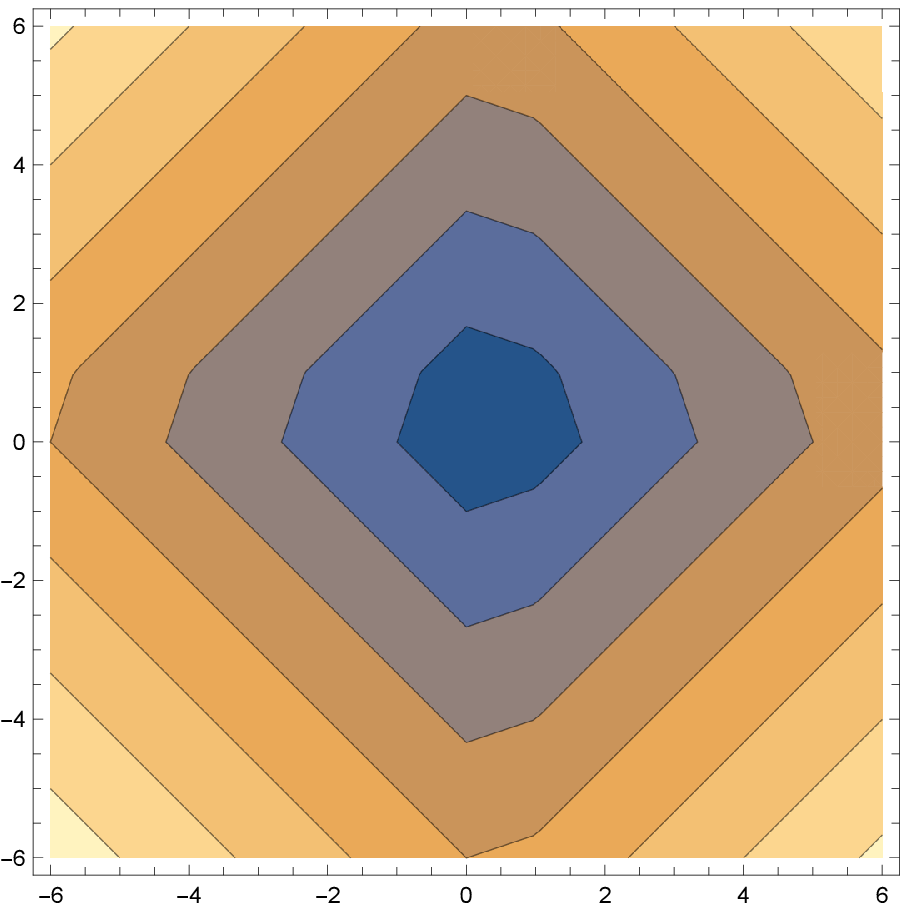} 
\caption{\Small The $3$-ellipse for
$x_{1}=\left( 1,0\right) ,$ $x_{2}=\left( 0,0\right) ,$ $x_{3}=\left(
0,1\right) $.} \label{fig:1}
\end{figure}

\begin{example}
\label{sekil2} Let $\left( X=%
%TCIMACRO{\U{211d} }%
%BeginExpansion
\mathbb{R}
%EndExpansion
^{2},d\right) $ be a metric space with the metric $d:X\times X\rightarrow 
%TCIMACRO{\U{211d} }%
%BeginExpansion
\mathbb{R}
%EndExpansion
$ defined as%
\begin{equation*}
d\left( a,b\right) =\sqrt{\left( u_{1}-u_{2}\right) ^{2}+\left(
v_{1}-v_{2}\right) ^{2}}\text{,}
\end{equation*}%
such that $a=\left( u_{1},v_{1}\right) ,b=\left( u_{2},v_{2}\right) \in X$.
Let us define $3$-ellipse for $x_{1}=\left( 3,0\right) ,$ $x_{2}=\left(
0,0\right) ,$ $x_{3}=\left( 0,4\right) $ as follows ($($see Figure \ref%
{fig:2}$):$%
\begin{equation*}
E\left[ x_{1},x_{2},x_{3};r\right] =\left\{ p\left( x,y\right) \in X:\sqrt{%
\left( x-3\right) ^{2}+y^{2}}+\sqrt{x^{2}+y^{2}}+\sqrt{x^{2}+\left(
y-1\right) ^{2}}=r\right\} .
\end{equation*}
\end{example}

\begin{figure}[h]
\centering
\includegraphics[width=.4\linewidth]{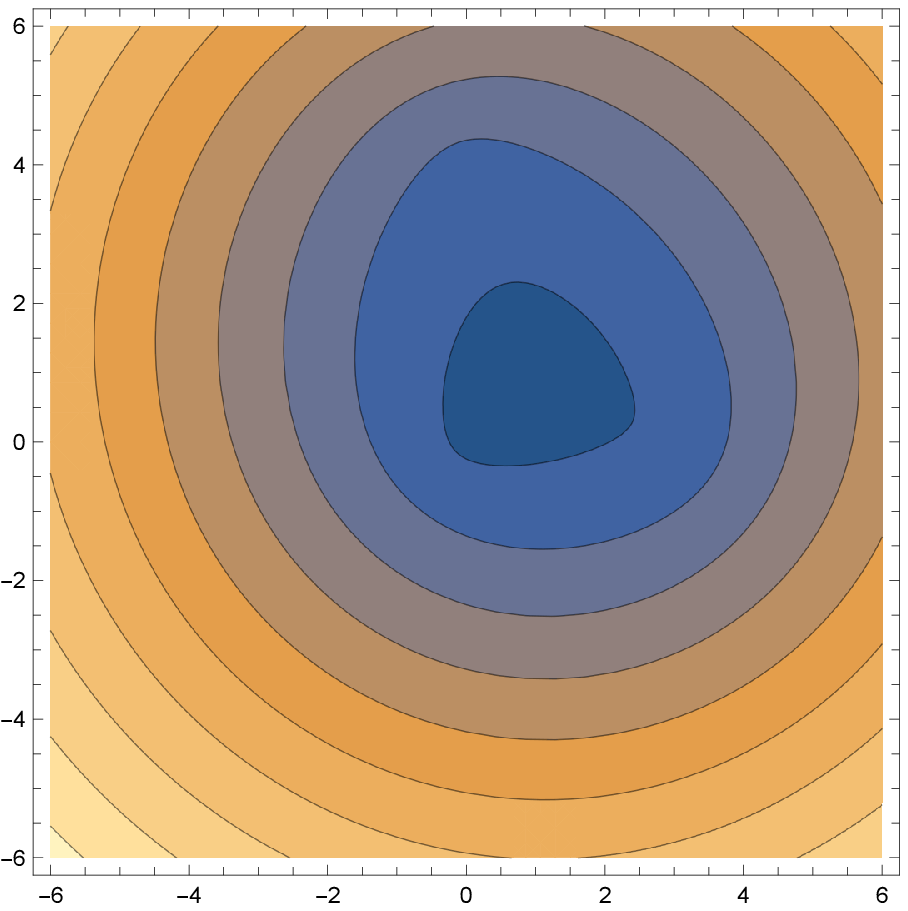} 
\caption{\Small The $3$-ellipse for
$x_{1}=\left( 3,0\right) ,$ $x_{2}=\left( 0,0\right) ,$ $x_{3}=\left(
0,4\right) $.} \label{fig:2}
\end{figure}

\begin{example}
\label{sekil3} Let $\left( X=%
%TCIMACRO{\U{211d} }%
%BeginExpansion
\mathbb{R}
%EndExpansion
^{2},d\right) $ be a metric space with the metric $d:X\times X\rightarrow 
%TCIMACRO{\U{211d} }%
%BeginExpansion
\mathbb{R}
%EndExpansion
$ defined as%
\begin{equation*}
d\left( a,b\right) =\max \left\{ \left\vert u_{1}-u_{2}\right\vert
+\left\vert v_{1}-v_{2}\right\vert \right\} ,
\end{equation*}%
such that $a=\left( u_{1},v_{1}\right) ,b=\left( u_{2},v_{2}\right) \in X.$
Let us define $3$-ellipse for $x_{1}=\left( 1,0\right) ,$ $x_{2}=\left(
0,0\right) ,$ $x_{3}=\left( 0,1\right) $ as follows $($see Figure \ref{fig:3}%
$):$%
\begin{equation*}
E\left[ x_{1},x_{2},x_{3};r\right] =\left\{ p\left( x,y\right) \in X:\max
\left\{ \left\vert x-1\right\vert ,\left\vert y\right\vert \right\} +\max
\left\{ \left\vert x\right\vert ,\left\vert y\right\vert \right\} +\max
\left\{ \left\vert x\right\vert ,\left\vert y-1\right\vert \right\}
=r\right\} .
\end{equation*}
\end{example}

\begin{figure}[h]
\centering
\includegraphics[width=.4\linewidth]{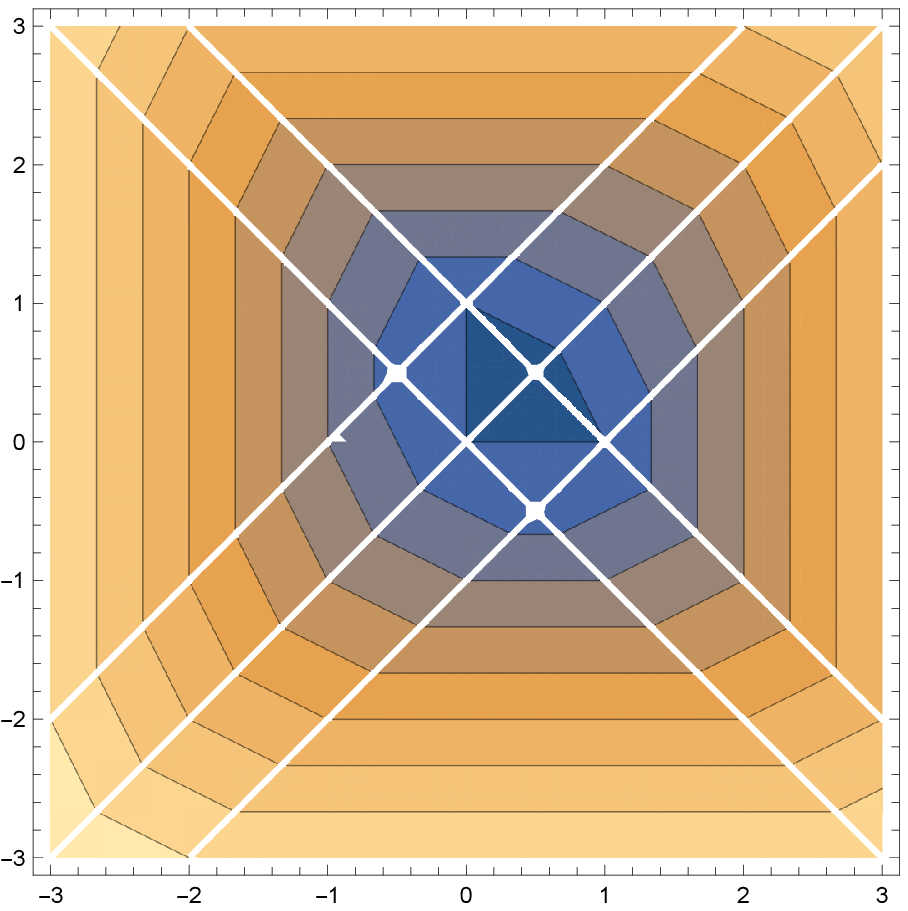} 
\caption{\Small The $3$-ellipse for
$x_{1}=\left( 1,0\right) ,$ $x_{2}=\left( 0,0\right) ,$ $x_{3}=\left(
0,1\right)$.} \label{fig:3}
\end{figure}

\begin{example}
\label{sekil4} Let $\left( X=%
%TCIMACRO{\U{211d} }%
%BeginExpansion
\mathbb{R}
%EndExpansion
^{3},d\right) $ be a metric space with the metric $d:X\times X\rightarrow 
%TCIMACRO{\U{211d} }%
%BeginExpansion
\mathbb{R}
%EndExpansion
$ defined as%
\begin{equation*}
d\left( a,b\right) =\sqrt{\left( u_{1}-u_{2}\right) ^{2}+\left(
v_{1}-v_{2}\right) ^{2}+\left( z_{1}-z_{2}\right) ^{2}},
\end{equation*}%
such that $a=\left( u_{1},v_{1},z_{1}\right) ,b=\left(
u_{2},v_{2},z_{2}\right) \in X.$ Let us define $3$-ellipse for $x_{1}=\left(
5,0,0\right) ,$ $x_{2}=\left( 0,2,0\right) ,$ $x_{3}=\left( 0,0,1\right) $
as follows $($see Figure \ref{fig:4}$):$%
\begin{equation*}
E\left[ x_{1},x_{2},x_{3};r\right] =\left\{ p\left( x,y\right) \in X:\sqrt{%
\left( x-5\right) ^{2}+y^{2}+z^{2}}+\sqrt{x^{2}+\left( y-2\right) ^{2}+z^{2}}%
\right.
\end{equation*}%
\begin{equation*}
\left. +\sqrt{x^{2}+y^{2}+\left( z-1\right) ^{2}}=r\right\} .
\end{equation*}
\end{example}

\begin{figure}[h]
\centering
\includegraphics[width=.4\linewidth]{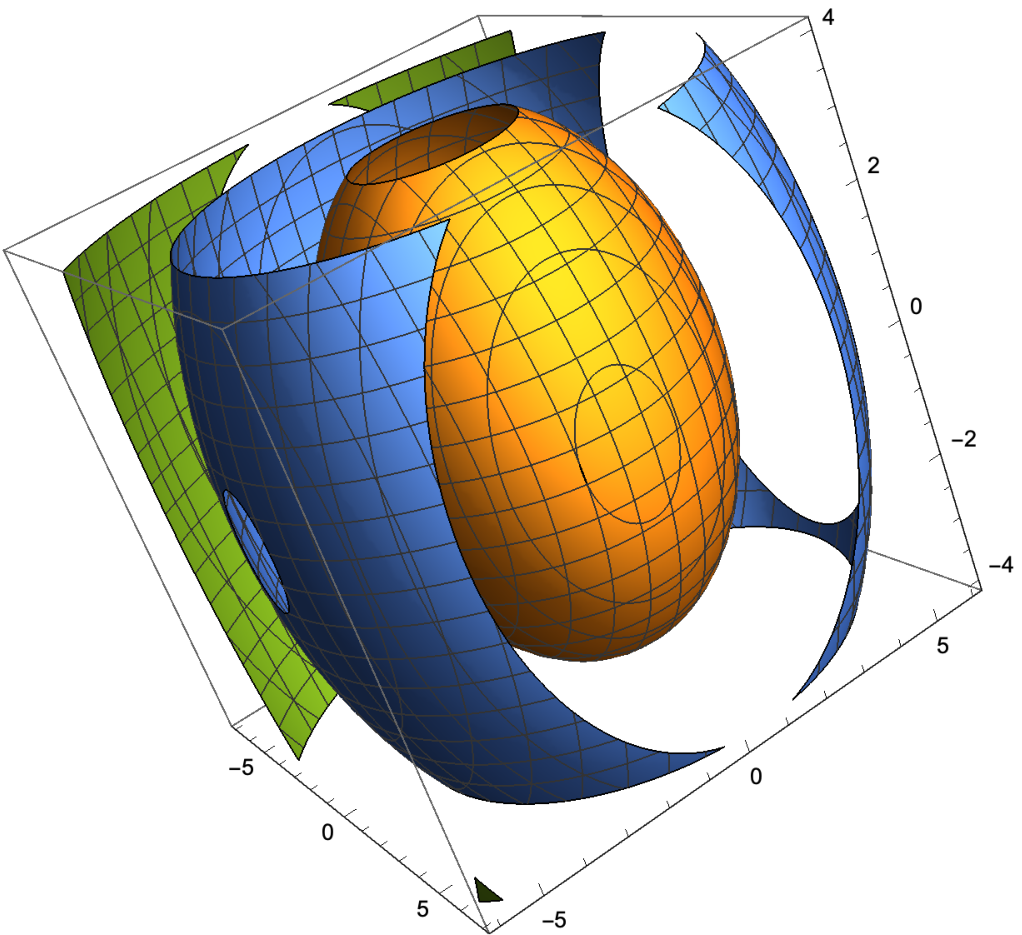} 
\caption{\Small The $3$-ellipse for
$x_{1}=\left( 5,0,0\right) ,$ $x_{2}=\left( 0,2,0\right) ,$ $x_{3}=\left(
0,0,1\right) $.} \label{fig:4}
\end{figure}

\begin{example}
\label{sekil5} Let $\left( X=%
%TCIMACRO{\U{211d} }%
%BeginExpansion
\mathbb{R}
%EndExpansion
^{3},d\right) $ be a metric space with the metric $d:X\times X\rightarrow 
%TCIMACRO{\U{211d} }%
%BeginExpansion
\mathbb{R}
%EndExpansion
$ defined as%
\begin{equation*}
d\left( a,b\right) =\sqrt[4]{\left( u_{1}-u_{2}\right) ^{4}+\left(
v_{1}-v_{2}\right) ^{4}+\left( z_{1}-z_{2}\right) ^{4}},
\end{equation*}%
such that $a=\left( u_{1},v_{1},z_{1}\right) ,b=\left(
u_{2},v_{2},z_{2}\right) \in X.$ Let us define $3$-ellipse for $x_{1}=\left(
-1,0,0\right) ,$ $x_{2}=\left( 1,0,0\right) ,$ $x_{3}=\left( 0,1,0\right) $
as follows $($see Figure \ref{fig:5}$):$%
\begin{equation*}
E\left[ x_{1},x_{2},x_{3};r\right] =\left\{ p\left( x,y\right) \in X:\sqrt[4]%
{\left( x+1\right) ^{4}+y^{4}+z^{4}}\right.
\end{equation*}%
\begin{equation*}
\left. +\sqrt[4]{\left( x-1\right) ^{4}+y^{4}+z^{4}}+\sqrt[4]{x^{4}+\left(
y-1\right) ^{4}+z^{4}}=r\right\} .
\end{equation*}
\end{example}

\begin{figure}[h]
\centering
\includegraphics[width=.4\linewidth]{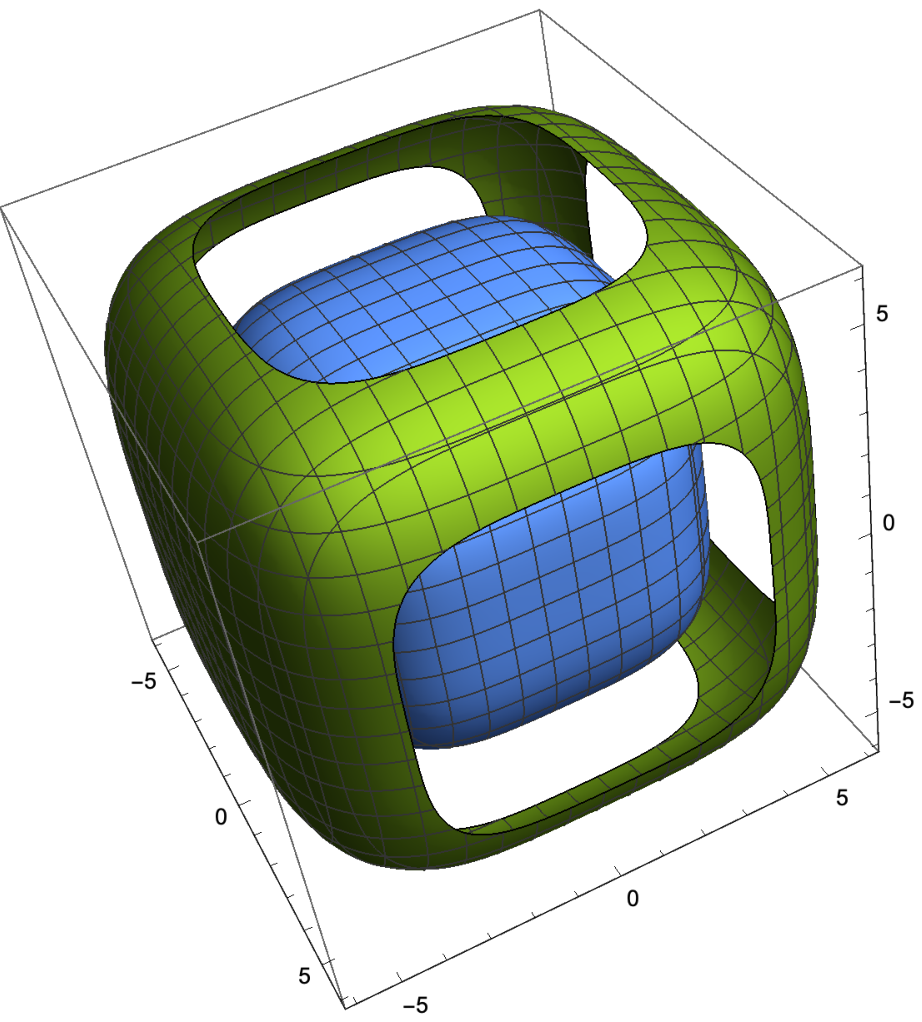} 
\caption{\Small The $3$-ellipse for
$x_{1}=\left( -1,0,0\right) ,$ $x_{2}=\left( 1,0,0\right) ,$ $x_{3}=\left(
0,1,0\right) $.} \label{fig:5}
\end{figure}

\begin{example}
\label{sekil6} Let $\left( X=%
%TCIMACRO{\U{211d} }%
%BeginExpansion
\mathbb{R}
%EndExpansion
^{2},d\right) $ be a metric space with the metric $d:X\times X\rightarrow 
%TCIMACRO{\U{211d} }%
%BeginExpansion
\mathbb{R}
%EndExpansion
$ defined as%
\begin{equation*}
d\left( a,b\right) =\sqrt{\left( u_{1}-u_{2}\right) ^{2}+\left(
v_{1}-v_{2}\right) ^{2}},
\end{equation*}%
such that $a=\left( u_{1},v_{1}\right) ,b=\left( u_{2},v_{2}\right) \in X.$
Let us define $4$-ellipse for $x_{1}=\left( 2,0\right) ,$ $x_{2}=\left(
0,0\right) ,$ $x_{3}=\left( 0,3\right) ,$ $x_{4}=\left( -2,0\right) $ as
follows $($see Figure \ref{fig:6}$):$%
\begin{equation*}
E\left[ x_{1},x_{2},x_{3},x_{4};r\right] =\left\{ p\left( x,y\right) \in X:%
\sqrt{\left( x-2\right) ^{2}+y^{2}}+\sqrt{x^{2}+y^{2}}\right.
\end{equation*}%
\begin{equation*}
\left. +\sqrt{x^{2}+\left( y-3\right) ^{2}}+\sqrt{\left( x+2\right)
^{2}+y^{2}}=r\right\} .
\end{equation*}
\end{example}

\begin{figure}[h]
\centering
\includegraphics[width=.4\linewidth]{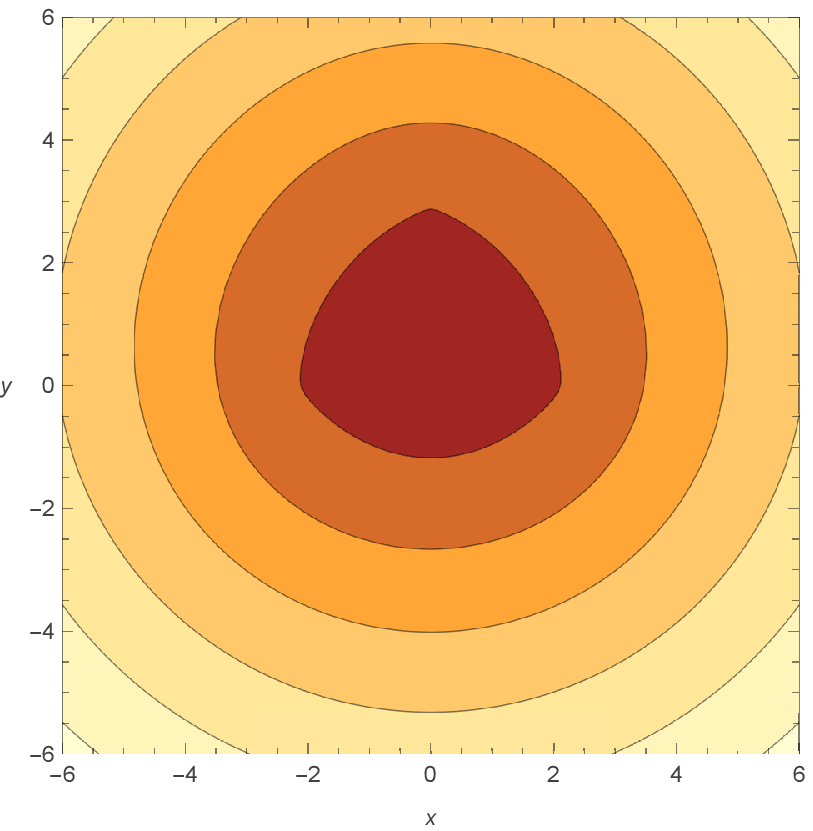} 
\caption{\Small The $4$-ellipse for $x_{1}=\left( 2,0\right) ,$ $x_{2}=\left( 0,0\right) ,$ $x_{3}=\left( 0,3\right) ,$ $x_{4}=\left(
-2,0\right) $.} \label{fig:6}
\end{figure}

We begin with the following proposition:

\begin{proposition}
\label{prop1} Let $\left( X,d\right) $ be a metric space and $E\left[
x_{1},...,x_{k};r\right] $, $E\left[ x_{1}^{^{\prime }},...,x_{k}^{^{\prime
}};r^{^{\prime }}\right] $ two $k$-ellipses. Then there exists at least one
of self-mapping $T:X\rightarrow X$ such that $T$ fixes the $k$-ellipses $E%
\left[ x_{1},...,x_{k};r\right] $ and$\ E\left[ x_{1}^{^{\prime
}},...,x_{k}^{^{\prime }};r^{^{\prime }}\right] .$
\end{proposition}

\begin{proof}
Let $E\left[ x_{1},...,x_{k};r\right] $ and$\ E\left[ x_{1}^{^{\prime
}},...,x_{k}^{^{\prime }};r^{^{\prime }}\right] $ be any $k$-ellipses on $X.$
Let us define the self-mapping $T:X\rightarrow X$ as 
\begin{equation*}
Tx=\left\{ 
\begin{array}{ccc}
x & , & x\in E\cup E^{^{\prime }} \\ 
\alpha & , & \text{otherwise}%
\end{array}%
,\right.
\end{equation*}%
for all $x\in X$, where $\alpha $ is a constant such that $%
\sum_{i=1}^{k}d\left( x,x_{i}\right) \neq r$ and $\sum_{i=1}^{k}d\left(
x,x_{i}^{^{\prime }}\right) \neq r^{^{\prime }}.$ It can be easily seen that 
$x\in Fix\left( T\right) $ for all $x\in E\cup E^{^{\prime }}$, that is, $T$
fixes the $k$-ellipses both $E$ and $E^{^{\prime }}.$
\end{proof}

\begin{remark}
\label{rmk1}

$1)$ Proposition \ref{prop1} generalizes Proposition $3.1$ give in \cite%
{Ozgur-malaysian} and Proposition $4$ give in \cite{Joshi}.

$2)$ Proposition \ref{prop1} can be extended as follows$:$

\ \ "Let $\left( X,d\right) $ be a metric space and $E\left[
x_{1},...,x_{k};r\right] ,..,E\left[ x_{1}^{n},...,x_{k}^{n};r^{n}\right] $
any $k$-ellipses. Then there exists at least one self-mapping $%
T:X\rightarrow X$ \ such that $T$ fixes the $k$-ellipses $E\left[
x_{1},...,x_{k};r\right] ,..,E\left[ x_{1}^{n},...,x_{k}^{n};r^{n}\right] .$"

From the above reasons, it is important to investigate the existence and
uniqueness conditions of the fixed $k$-ellipses.,
\end{remark}

Now, we give some theorems about the existence and uniqueness conditions of
the fixed $k$-ellipses.

\begin{theorem}
\label{thrm1} Let $\left( X,d\right) $ be a metric space and $E\left[
x_{1},...,x_{k};r\right] $ any $k$-ellipse on $X$. Let us define the mapping 
$\xi :X\rightarrow \left[ 0,\infty \right) $ as%
\begin{equation*}
\xi \left( x\right) =\sum_{i=1}^{k}d\left( x,x_{i}\right) ,
\end{equation*}%
for all $x\in X.$ If there exists a self-mapping $T:X\rightarrow X$ such that

$\left( E_{k}1\right) $ $d\left( x,Tx\right) \leq \xi \left( x\right) -\xi
\left( Tx\right) $ \ \ for each $x\in E\left[ x_{1},...,x_{k};r\right] ,$

$\left( E_{k}2\right) $ $\sum_{i=1}^{k}d\left( Tx,x_{i}\right) \geq r$ \ \ \
\ \ for each $x\in E\left[ x_{1},...,x_{k};r\right] ,$

$\left( E_{k}3\right) $ $d\left( Tx,Ty\right) \leq h\left[ d\left(
Tx,x\right) +d\left( Ty,y\right) \right] $ \ \ for each $x\in E\left[
x_{1},...,x_{k};r\right] ,$ $y\in X-E\left[ x_{1},...,x_{k};r\right] $ and
some $h\in \left( 0,\frac{1}{2}\right) ,$\newline
then $E\left[ x_{1},...,x_{k};r\right] $ is a unique fixed $k$-ellipse of $%
T. $
\end{theorem}

\begin{proof}
At first, we prove the existence of a fixed $k$-ellipse of $T.$ Let $x\in E%
\left[ x_{1},...,x_{k};r\right] $ be any point. Using the conditions $\left(
E_{k}1\right) ,$ $\left( E_{k}2\right) $ and the definition of $\xi ,$ we get%
\begin{eqnarray*}
d\left( x,Tx\right) &\leq &\xi \left( x\right) -\xi \left( Tx\right) \\
&=&\sum_{i=1}^{k}d\left( x,x_{i}\right) -\sum_{i=1}^{k}d\left(
Tx,x_{i}\right) \\
&=&r-\sum_{i=1}^{k}d\left( Tx,x_{i}\right) \\
&\leq &r-r=0,
\end{eqnarray*}%
that is, $x\in Fix\left( T\right) =\left\{ x\in X:x=Tx\right\} .$ Hence, $E%
\left[ x_{1},...,x_{k};r\right] $ is a fixed $k$-ellipse of $T.$

Now, we show that $E\left[ x_{1},...,x_{k};r\right] $ is a unique fixed $k$%
-ellipse of $T.$ To do this, on the contrary, we suppose that $E\left[
x_{1}^{^{\prime }},...,x_{k}^{^{\prime }};r^{^{\prime }}\right] $ is another
fixed $k$-ellipse of $T.$ Let $x\in E\left[ x_{1},...,x_{k};r\right] $ and $%
y\in E\left[ x_{1}^{^{\prime }},...,x_{k}^{^{\prime }};r^{^{\prime }}\right] 
$ such that $x\neq y.$ Then using the condition $\left( E_{k}3\right) $, we
obtain%
\begin{equation*}
d\left( Tx,Ty\right) =d\left( x,y\right) \leq h\left[ d\left( Tx,x\right)
+d\left( Ty,y\right) \right] =h\left[ d\left( x,x\right) +d\left( y,y\right) %
\right] =0,
\end{equation*}%
a contradiction with $x\neq y.$ It should be $x=y.$

Consequently $E\left[ x_{1},...,x_{k};r\right] $ is a unique fixed $k$%
-ellipse of $T.$
\end{proof}

If we consider the above theorem, we obtain the following remark:

\begin{remark}
\label{rmk2}

$1)$ If $k=1,$ then we have $E\left[ x_{1};r\right] =C_{x_{1},r}$ and so by
Theorem $2.1$ given in \cite{Ozgur-malaysian}, $E\left[ x_{1};r\right] $ is
a fixed $1$-ellipse of $T$ or $C_{x_{1},r}$ is a fixed circle of $T.$

$2)$ If $k=2,$ then we have $E\left[ x_{1},x_{2};r\right] =E_{r}\left(
x_{1},x_{2}\right) $ and so by Theorem $1$ given in \cite{Joshi}, $E\left[
x_{1},x_{2};r\right] $ is a fixed $2$-ellipse of $T$ or $E_{r}\left(
x_{1},x_{2}\right) $ is a fixed ellipse of $T.$

$3)$ The condition $\left( E_{k}3\right) $ can be changed with a proper
contractive condition such as%
\begin{equation*}
d\left( Tx,Ty\right) \leq hd\left( x,y\right) ,
\end{equation*}%
where $h\in \left( 0,1\right) .$ This contraction can be considered as
Banach type contractive condition \cite{Banach}.

$4)$ If we pay attention, the condition $\left( E_{k}1\right) $ can be
considered as Caristi type contractive condition \cite{Caristi} and the
condition $\left( E_{k}3\right) $ can be considered as Kannan type
contractive condition \cite{Kannan}.

$5)$ The condition $\left( E_{k}1\right) $ guarantees that $Tx$ is not in
the exterior of the $k$-ellipse $E\left[ x_{1},...,x_{k};r\right] $ for each 
$x\in E\left[ x_{1},...,x_{k};r\right] $ and the condition $\left(
E_{k}2\right) $ guarantees that $Tx$ is not in the interior of the $k$%
-ellipse $E\left[ x_{1},...,x_{k};r\right] $ for each $x\in E\left[
x_{1},...,x_{k};r\right] .$ Therefore, we say $T\left( E\left[
x_{1},...,x_{k};r\right] \right) \subset E\left[ x_{1},...,x_{k};r\right] .$
\end{remark}

Now, we give the examples which satisfy the above theorem:

\begin{example}
\label{exmp1} Let $\left( X,d\right) $ be a metric space, $E=E\left[
x_{1},...,x_{k};r\right] $ any $k$-ellipse and $z$ a constant such that%
\begin{equation*}
2d\left( x,z\right) <d\left( y,z\right) ,
\end{equation*}%
for all $x\in E$ and $y\in X-E.$

Let us define the self-mapping $T:X\rightarrow X$ as%
\begin{equation*}
Tx=\left\{ 
\begin{array}{ccc}
x & , & \text{\ }x\in E \\ 
z & , & \text{\ otherwise}%
\end{array}%
,\right.
\end{equation*}%
for all $x\in X.$ Then it is clear that $T$ satisfies the conditions $\left(
E_{k}1\right) $ and $\left( E_{k}2\right) $ and so $E\left[ x_{1},...,x_{k};r%
\right] $ is a fixed $k$-ellipse of $T.$ Now we show that $T$ satisfies the
condition $\left( E_{k}3\right) .$ Let $x\in E\left[ x_{1},...,x_{k};r\right]
$ and $y\in X-E\left[ x_{1},...,x_{k};r\right] .$ Then we have 
\begin{eqnarray*}
d\left( Tx,Ty\right) &=&d\left( x,z\right) \leq h\left[ d\left( Tx,x\right)
+d\left( Ty,y\right) \right] \\
\text{ \ \ \ \ \ \ \ \ \ \ \ } &=&h\left[ d\left( x,x\right) +d\left(
z,y\right) \right] =hd\left( z,y\right) ,
\end{eqnarray*}%
with $h\in \left( 0,\frac{1}{2}\right) .$ Therefore, $E\left[
x_{1},...,x_{k};r\right] $ is a unique fixed $k$-ellipse of $T.$
\end{example}

In the following example, we see that the fixed $k$-ellipse of $T$ does not
have to be unique:

\begin{example}
\label{exmp2} Let $\left( X,d\right) $ be a metric space, $E_{1}=E\left[
x_{1},...,x_{k};r\right] ,$ $E_{2}=E\left[ x_{1}^{^{\prime
}},...,x_{k}^{^{\prime }};r^{^{\prime }}\right] $ any two $k$-ellipses and $%
z $ a constant such that%
\begin{equation*}
\sum_{i=1}^{k}d\left( x,x_{i}\right) \neq r\text{ and }\sum_{i=1}^{k}d\left(
x,x_{i}^{^{\prime }}\right) \neq r^{^{\prime }}.
\end{equation*}%
Let us define the self-mapping $T:X\rightarrow X$ as%
\begin{equation*}
Tx=\left\{ 
\begin{array}{ccc}
x & , & x\in E_{1}\cup E_{2} \\ 
z & ,\text{\ } & \text{\ otherwise}%
\end{array}%
,\right.
\end{equation*}%
for all $x\in X.$ Then $T$ satisfies the conditions $\left( E_{k}1\right) $
and $\left( E_{k}2\right) $ for both $x\in E_{1}$ and $x\in E_{2}.$ Hence $T$
fixes both $E_{1}$ and $E_{2}.$ But $T$ does not satisfy the condition $%
\left( E_{k}3\right) .$ Indeed, for $x\in E_{1}$ and $y\in E_{2}$ with $%
x\neq y,$ we have%
\begin{equation*}
d\left( Tx,Ty\right) =d\left( x,y\right) \leq h\left[ d\left( Tx,x\right)
+d\left( Ty,y\right) \right] =0,
\end{equation*}%
a contradiction with $x\neq y.$

Consequently, the fixed $k$-ellipse of $T$ is not unique.
\end{example}

\begin{example}
\label{exmp3} Let $X=%
%TCIMACRO{\U{211d} }%
%BeginExpansion
\mathbb{R}
%EndExpansion
$ be the usual metric space with the usual metric $d$ defined as%
\begin{equation*}
d\left( x,y\right) =\left\vert x-y\right\vert ,
\end{equation*}%
$x,y\in 
%TCIMACRO{\U{211d} }%
%BeginExpansion
\mathbb{R}
%EndExpansion
.$ Let us take a $3$-ellipse $E\left[ -1,0,1;9\right] =\left\{ x\in 
%TCIMACRO{\U{211d} }%
%BeginExpansion
\mathbb{R}
%EndExpansion
:d\left( x,-1\right) +d\left( x,0\right) +d\left( x,1\right) =9\right\} $
and define the self-mapping $T:$ $%
%TCIMACRO{\U{211d} }%
%BeginExpansion
\mathbb{R}
%EndExpansion
\rightarrow 
%TCIMACRO{\U{211d} }%
%BeginExpansion
\mathbb{R}
%EndExpansion
$ as%
\begin{equation*}
Tx=\left\{ 
\begin{array}{ccc}
0 & , & x\in \left\{ -3,3\right\} \\ 
-3 & , & x\in 
%TCIMACRO{\U{211d} }%
%BeginExpansion
\mathbb{R}
%EndExpansion
-\left\{ -3,3\right\}%
\end{array}%
\right. ,
\end{equation*}%
for all $x\in 
%TCIMACRO{\U{211d} }%
%BeginExpansion
\mathbb{R}
%EndExpansion
.$ Then $T$ satisfies the condition $\left( E_{k}1\right) $ but does not
satisfy the condition $\left( E_{k}2\right) .$ Therefore, $E\left[ -1,0,1;9%
\right] $ is not a fixed $3$-ellipse of $T.$

On the other hand, let us define the self-mapping $S:$ $%
%TCIMACRO{\U{211d} }%
%BeginExpansion
\mathbb{R}
%EndExpansion
\rightarrow 
%TCIMACRO{\U{211d} }%
%BeginExpansion
\mathbb{R}
%EndExpansion
$ as%
\begin{equation*}
Sx=\left\{ 
\begin{array}{ccc}
5 & , & x\in \left\{ -3,3\right\} \\ 
0 & , & x\in 
%TCIMACRO{\U{211d} }%
%BeginExpansion
\mathbb{R}
%EndExpansion
-\left\{ -3,3\right\}%
\end{array}%
,\right.
\end{equation*}%
for all $x\in 
%TCIMACRO{\U{211d} }%
%BeginExpansion
\mathbb{R}
%EndExpansion
.$ Then $S$ satisfies the condition $\left( E_{k}2\right) $ but does not
satisfy the condition $\left( E_{k}1\right) .$ Hence $E\left[ -1,0,1;9\right]
$ is not a fixed $3$-ellipse of $S.$
\end{example}

\begin{theorem}
\label{thrm2} Let $\left( X,d\right) $ be a metric space and $E\left[
x_{1},...,x_{k};r\right] $ any $k$-ellipse on $X$ and the mapping $\xi
:X\rightarrow \left[ 0,\infty \right) $ defined as in Theorem \ref{thrm1}.
If there exists a self-mapping $T:X\rightarrow X$ such that

$\left( E_{k}^{^{\prime }}1\right) $ $d\left( x,Tx\right) \leq \xi \left(
x\right) +\xi \left( Tx\right) -2r$ \ \ for each $x\in E\left[
x_{1},...,x_{k};r\right] ,$

$\left( E_{k}^{^{\prime }}2\right) $ $\sum_{i=1}^{k}d\left( Tx,x_{i}\right)
\leq r$ \ \ \ \ \ for each $x\in E\left[ x_{1},...,x_{k};r\right] ,$

$\left( E_{k}^{^{\prime }}3\right) $ $d\left( Tx,Ty\right) \leq h\left[
d\left( Tx,y\right) +d\left( Ty,x\right) \right] $ \ \ for each $x\in E\left[
x_{1},...,x_{k};r\right] ,$ $y\in X-E\left[ x_{1},...,x_{k};r\right] $ and
some $h\in \left[ 0,\frac{1}{2}\right) ,$\newline
then $E\left[ x_{1},...,x_{k};r\right] $ is a unique fixed $k$-ellipse of $%
T. $
\end{theorem}

\begin{proof}
At first, we prove the existence of a fixed $k$-ellipse of $T.$ Let $x\in E%
\left[ x_{1},...,x_{k};r\right] $ be any point. Using the conditions $\left(
E_{k}^{^{\prime }}1\right) ,$ $\left( E_{k}^{^{\prime }}2\right) $ and the
definition of $\xi $, we get%
\begin{eqnarray*}
d\left( x,Tx\right) &\leq &\xi \left( x\right) +\xi \left( Tx\right) -2r \\
&=&\sum_{i=1}^{k}d\left( x,x_{i}\right) +\sum_{i=1}^{k}d\left(
Tx,x_{i}\right) -2r \\
&=&r+\sum_{i=1}^{k}d\left( Tx,x_{i}\right) -2r \\
&=&\sum_{i=1}^{k}d\left( Tx,x_{i}\right) -r\leq r-r=0,
\end{eqnarray*}%
that is, $x\in Fix(T).$ So $E\left[ x_{1},...,x_{k};r\right] $ is a fixed $k$%
-ellipse of $T.$

Now, we show that $E\left[ x_{1},...,x_{k};r\right] $ is a unique fixed $k$%
-ellipse of $T.$ To do this, on the contrary, we assume that $E\left[
x_{1}^{^{\prime }},...,x_{k}^{^{\prime }};r^{^{\prime }}\right] $ is another
fixed $k$-ellipse of $T.$ Let $x\in E\left[ x_{1},...,x_{k};r\right] $ and $%
y\in E\left[ x_{1}^{^{\prime }},...,x_{k}^{^{\prime }};r^{^{\prime }}\right] 
$ such that $x\neq y.$ Then using the condition $\left( E_{k}^{^{\prime
}}3\right) ,$ we find%
\begin{equation*}
d\left( Tx,Ty\right) =d\left( x,y\right) \leq h\left[ d\left( Tx,y\right)
+d\left( Ty,x\right) \right] =h\left[ d\left( x,y\right) +d\left( y,x\right) %
\right] =2hd\left( x,y\right) ,
\end{equation*}%
a contradiction with $h\in \left[ 0,\frac{1}{2}\right) .$ It should be $x=y.$
Consequently, $E\left[ x_{1},...,x_{k};r\right] $ is a unique fixed $k$%
-ellipse of $T.$
\end{proof}

\begin{remark}
\label{rmk3}

$1)$ If $k=1,$ then we have $E\left[ x_{1};r\right] =C_{x_{1},r}$ and so by
Theorem $2.2$ given in \cite{Ozgur-malaysian}, $E\left[ x_{1};r\right] $ is
a fixed $1$-ellipse of $T$ or $C_{x_{1},r}$ is a fixed circle of $T.$

$2)$ If $k=2,$ then we consider Theorem \ref{thrm2} as a fixed-ellipse
theorem.

$3)$ The determination of the condition $\left( E_{k}^{^{\prime }}3\right) $
is not unique.

$4)$ If we pay close attention, the condition $\left( E_{k}^{^{\prime
}}3\right) $ can be considered as a Chatterjea type contractive condition 
\cite{Chatterjea}.

$5)$ The condition $\left( E_{k}^{^{\prime }}1\right) $ guarantees that $Tx$
is not in the interior of the $k$-ellipse $E\left[ x_{1},...,x_{k};r\right] $
for each $x\in E\left[ x_{1},...,x_{k};r\right] $ and the condition $\left(
E_{k}^{^{\prime }}2\right) $ guarantees that $Tx$ is not in the exterior of
the $k$-ellipse $E\left[ x_{1},...,x_{k};r\right] $ for each $x\in E\left[
x_{1},...,x_{k};r\right] .$ Hence, we say $T\left( E\left[ x_{1},...,x_{k};r%
\right] \right) \subset E\left[ x_{1},...,x_{k};r\right] .$
\end{remark}

\begin{example}
\label{exmp4} Let $X=\left\{ -4,-1,0,1,2,18\right\} $ be the usual metric
space. Let us take a $4$-ellipse $E\left[ -1,0,1,2;18\right] $ such as%
\begin{eqnarray*}
E\left[ -1,0,1,2;18\right] &=&\left\{ x\in X:\left\vert x+1\right\vert
+\left\vert x\right\vert +\left\vert x-1\right\vert +\left\vert
x-2\right\vert =18\right\} \\
&=&\left\{ -4\right\} .
\end{eqnarray*}%
Let us define the self-mapping $T:X\rightarrow X$ as%
\begin{equation*}
Tx=\left\{ 
\begin{array}{ccc}
-4 & , & x\in X-\left\{ -1\right\} \\ 
\text{\ }0 & , & x=-1%
\end{array}%
,\right.
\end{equation*}%
for all $x\in X.$ Then $T$ satisfies the conditions $\left( E_{k}^{^{\prime
}}1\right) $ and $\left( E_{k}^{^{\prime }}2\right) $ and so $E\left[
-1,0,1,2;18\right] $ is a fixed $4$-ellipse of $T.$ Also $T$ satisfies the
condition $\left( E_{k}^{^{\prime }}3\right) $ with $h=\frac{4}{9}.$
Consequently, $E\left[ -1,0,1,2;18\right] $ is a unique fixed $4$-ellipse of 
$T.$
\end{example}

\begin{theorem}
\label{thrm3} Let $\left( X,d\right) $ be a metric space and $E\left[
x_{1},...,x_{k};r\right] $ any $k$-ellipse on $X$ and the mapping $\xi
:X\rightarrow \left[ 0,\infty \right) $ defined as in Theorem \ref{thrm1}.
If there exists a self-mapping $T:X\rightarrow X$ satisfying the conditions $%
\left( E_{k}1\right) $, $\left( E_{k}3\right) $ and

$\left( E_{k}^{^{\prime \prime }}2\right) $ $\mu d\left( x,Tx\right) +$ $%
\sum_{i=1}^{k}d\left( Tx,x_{i}\right) \geq r$ for each $x\in E\left[
x_{1},...,x_{k};r\right] $ and some $\mu \in \left[ 0,1\right) ,$\newline
then $E\left[ x_{1},...,x_{k};r\right] $ is a unique fixed $k$-ellipse of $%
T. $
\end{theorem}

\begin{proof}
Using the similar approaches given in the proof of Theorem \ref{thrm1}, we
can easily prove that $E\left[ x_{1},...,x_{k};r\right] $ is a unique fixed $%
k$-ellipse of $T.$
\end{proof}

\begin{remark}
\label{rmk4}

$1)$ If $k=1,$ then by Theorem $2.3$ given in \cite{Ozgur-malaysian}, $E%
\left[ x_{1};r\right] $ is a fixed $1$-ellipse of $T$ or $C_{x_{1},r}$ is a
fixed circle of $T.$

$2)$ If $k=2,$ then we consider Theorem $2$ given in \cite{Joshi}, $E\left[
x_{1},x_{2};r\right] $ as a fixed $2$-ellipse of $T$ or $E_{r}\left(
x_{1},x_{2}\right) $ is a fixed ellipse of $T.$

$3)$ The condition $\left( E_{k}^{^{^{\prime \prime }}}2\right) $ implies
that $Tx$ can be lies on or exterior or interior of the $k$-ellipse $E\left[
x_{1},...,x_{k};r\right] $.
\end{remark}

\begin{remark}
\label{rmk5} If we consider Example \ref{exmp2}, then $T$ satisfies the
conditions $\left( E_{k}^{^{\prime }}1\right) $ and $\left( E_{k}^{^{\prime
}}2\right) $ but $T$ does not satisfy the condition $\left( E_{k}^{^{\prime
}}3\right) .$ Also, if we consider Example \ref{exmp3}, then $T$ satisfies
the condition $\left( E_{k}^{^{\prime }}2\right) $ but $T$ does not satisfy
the condition $\left( E_{k}^{^{\prime }}1\right) .$ Similarly, in the same
example, if we define the self-mapping $H:$ $%
%TCIMACRO{\U{211d} }%
%BeginExpansion
\mathbb{R}
%EndExpansion
\rightarrow 
%TCIMACRO{\U{211d} }%
%BeginExpansion
\mathbb{R}
%EndExpansion
$ as%
\begin{equation*}
Hx=\left\{ 
\begin{array}{ccc}
10 & , & \text{\ }x\in \left\{ -3,3\right\} \\ 
0 & \text{\ }, & \text{\ \ }x\in 
%TCIMACRO{\U{211d} }%
%BeginExpansion
\mathbb{R}
%EndExpansion
-\left\{ -3,3\right\}%
\end{array}%
,\right.
\end{equation*}%
for all $x\in 
%TCIMACRO{\U{211d} }%
%BeginExpansion
\mathbb{R}
%EndExpansion
$. Then $H$ satisfies the condition $\left( E_{k}^{^{\prime }}1\right) $ but
does not satisfy the condition $\left( E_{k}^{^{\prime }}2\right) $ .

Finally, if we consider \ Example \ref{exmp1} and Example \ref{exmp2}, then $%
T$ satisfies the conditions of $\left( E_{k}^{^{^{\prime \prime }}}2\right)
. $
\end{remark}

The selection of the auxiliary function is not unique. For example, we give
the following fixed $k$-ellipse theorem:

\begin{theorem}
\label{thrm4} Let $\left( X,d\right) $ be a metric space and $E\left[
x_{1},...,x_{k};r\right] $ any $k$-ellipse on $X$ . Let us define the
mapping $\psi :%
%TCIMACRO{\U{211d} }%
%BeginExpansion
\mathbb{R}
%EndExpansion
^{+}\cup \left\{ 0\right\} \rightarrow 
%TCIMACRO{\U{211d} }%
%BeginExpansion
\mathbb{R}
%EndExpansion
$ as%
\begin{equation*}
\psi \left( x\right) =\left\{ 
\begin{array}{ccc}
x-r & ,\text{ } & \text{\ }x>0 \\ 
0 & , & x=0%
\end{array}%
\right.
\end{equation*}%
for all $x\in 
%TCIMACRO{\U{211d} }%
%BeginExpansion
\mathbb{R}
%EndExpansion
^{+}\cup \left\{ 0\right\} $. If there exists a self-mapping $T:X\rightarrow
X$ satisfying

$\left( E_{k}^{^{\prime \prime \prime }}1\right) $ $\sum_{i=1}^{k}d\left(
Tx,x_{i}\right) =r$ for each $x\in E\left[ x_{1},...,x_{k};r\right] ,$

$\left( E_{k}^{^{\prime \prime \prime }}2\right) $ $d\left( Tx,Ty\right) >r$
for each $x,y\in E\left[ x_{1},...,x_{k};r\right] $ with $x\neq y,$

$\left( E_{k}^{^{\prime \prime \prime }}3\right) $ $d\left( Tx,Ty\right)
\leq d\left( x,y\right) -\psi \left( d\left( x,Tx\right) \right) $ for each $%
x,y\in E\left[ x_{1},...,x_{k};r\right] ,$

$\left( E_{k}^{^{\prime \prime \prime }}4\right) $ $d\left( Tx,Ty\right)
\leq h\max \left\{ d\left( x,Tx\right) ,d\left( y,Ty\right) ,d\left(
x,Ty\right) ,d\left( y,Tx\right) ,d\left( x,y\right) \right\} $ for each $%
x\in E\left[ x_{1},...,x_{k};r\right] ,$ $y\in X-E\left[ x_{1},...,x_{k};r%
\right] $ and some $h\in \left( 0,1\right) ,$\newline
then $E\left[ x_{1},...,x_{k};r\right] $ is a unique fixed $k$-ellipse of $%
T. $
\end{theorem}

\begin{proof}
To show the existence of a fixed $k$-ellipse of $T,$ we assume that $x\in E%
\left[ x_{1},...,x_{k};r\right] $ is any point. By the condition $\left(
E_{k}^{^{\prime \prime \prime }}1\right) $ ,we say that $Tx\in E\left[
x_{1},...,x_{k};r\right] $ for each $x\in E\left[ x_{1},...,x_{k};r\right] .$
Now we prove $x\in Fix\left( T\right) .$ On the contrary, let $x\notin
Fix\left( T\right) ,$ that is, $x\neq Tx.$ Using the condition $\left(
E_{k}^{^{\prime \prime \prime }}2\right) ,$ we get%
\begin{equation}
d\left( Tx,T^{2}x\right) >r  \label{metrik1}
\end{equation}%
and using the condition $\left( E_{k}^{^{\prime \prime \prime }}3\right) $,
we obtain%
\begin{eqnarray*}
d\left( Tx,T^{2}x\right)  &\leq &d\left( x,Tx\right) -\psi \left( d\left(
x,Tx\right) \right)  \\
&=&d\left( x,Tx\right) -d\left( x,Tx\right) +r=r,
\end{eqnarray*}%
a contradiction with the inequality (\ref{metrik1}). Thereby, it should be $%
x\in Fix\left( T\right) $ and $E\left[ x_{1},...,x_{k};r\right] $ is a fixed 
$k$-ellipse of $T.$

Finally, we prove the uniqueness of the fixed $k$-ellipse $E\left[
x_{1},...,x_{k};r\right] $ of $T.$ On the contrary, let $E\left[
x_{1}^{^{\prime }},...,x_{k}^{^{\prime }};r^{^{\prime }}\right] $ be another
fixed $k$-ellipse of $T,$ $x\in E\left[ x_{1},...,x_{k};r\right] $ and $y\in
E\left[ x_{1}^{^{\prime }},...,x_{k}^{^{\prime }};r^{^{\prime }}\right] $
such that $x\neq y.$ Using the condition $\left( E_{k}^{^{\prime \prime
\prime }}4\right) ,$ we find%
\begin{eqnarray*}
d\left( Tx,Ty\right) &\leq &h\max \left\{ d\left( x,Tx\right) ,d\left(
y,Ty\right) ,d\left( x,Ty\right) ,d\left( y,Tx\right) ,d\left( x,y\right)
\right\} \\
d\left( x,y\right) &=&h\max \left\{ 0,0,d\left( x,y\right) ,d\left(
y,x\right) ,d\left( x,y\right) \right\} \\
&=&hd\left( x,y\right) ,
\end{eqnarray*}%
a contradiction with $h\in \left( 0,1\right) .$ It should be $x=y.$
Consequently, $E\left[ x_{1},...,x_{k};r\right] $ is a unique fixed $k$%
-ellipse of $T.$
\end{proof}

\begin{remark}
\label{rmk6}

$1)$ If $k=1,$ then by Theorem $3$ given in \cite{Ozgur-Aip}, $E\left[
x_{1};r\right] $ is a fixed $1$-ellipse of $T$ or $C_{x_{1},r}$ is a fixed
circle of $T.$

$2)$ If $k=2,$ then we consider Theorem $4$ as a new fixed-ellipse result.

$3)$ The condition $\left( E_{k}^{^{\prime \prime \prime }}4\right) $ can be
considered as \'{C}iri\'{c} type contractive condition \cite{Ciric} and the
selection is not unique.

$4)$ The converse statements of Theorem \ref{thrm1}, Theorem \ref{thrm2} and
Theorem \ref{thrm3} are also true but the converse statement of Theorem \ref%
{thrm4} is not always true.
\end{remark}

\begin{example}
\label{exmp5} Let $X=\left\{ -1,0,1,4,12\right\} $ be the usual metric
space. Let us take a $3$-ellipse $E\left[ -1,0,1;12\right] $ such as%
\begin{equation*}
E\left[ -1,0,1;12\right] =\left\{ x\in 
%TCIMACRO{\U{211d} }%
%BeginExpansion
\mathbb{R}
%EndExpansion
:\left\vert x+1\right\vert +\left\vert x\right\vert +\left\vert
x-1\right\vert =12\right\} =\left\{ 4\right\} .
\end{equation*}%
Let us define the self-mapping $T:X\rightarrow X$ as%
\begin{equation*}
Tx=4,
\end{equation*}%
for all $x\in X.$ Then satisfies the conditions $\left( E_{k}^{^{\prime
\prime \prime }}1\right) ,$ $\left( E_{k}^{^{\prime \prime \prime }}2\right) 
$ and $\left( E_{k}^{^{\prime \prime \prime }}3\right) .$ So $E\left[
-1,0,1;12\right] $ is a fixed $3$-ellipse of $T.$ Also $T$ satisfy the
condition $\left( E_{k}^{^{\prime \prime \prime }}4\right) $ and we say that 
$E\left[ -1,0,1;12\right] $ is a unique fixed $3$-ellipse of $T.$
\end{example}

\begin{example}
\label{exmp6} Let $X=\left\{ -2.-1\right\} \cup \left[ 0,\infty \right) $
and $\left( X,d\right) $ be a usual metric space. Let us take a $3$-ellipse $%
E\left[ -2,0,2;21\right] $ such as%
\begin{equation*}
E\left[ -2,0,2;21\right] =\left\{ x\in 
%TCIMACRO{\U{211d} }%
%BeginExpansion
\mathbb{R}
%EndExpansion
:\left\vert x+2\right\vert +\left\vert x\right\vert +\left\vert
x-2\right\vert =21\right\} =\left\{ 7\right\} .
\end{equation*}%
Let us define the self-mapping $T:X\rightarrow X$ as%
\begin{equation*}
Tx=\left\{ 
\begin{array}{ccc}
x & , & x\in \left[ 0,\infty \right) \\ 
0 & , & x\in \left\{ -2,-1\right\}%
\end{array}%
,\right.
\end{equation*}%
for all $x\in X$. Then $T$ satisfies the conditions $\left( E_{k}^{^{\prime
\prime \prime }}1\right) ,$ $\left( E_{k}^{^{\prime \prime \prime }}2\right) 
$ and $\left( E_{k}^{^{\prime \prime \prime }}3\right) .$ Hence $E\left[
-2,0,2;21\right] $ is a fixed $3$-ellipse of $T.$ But $T$ does not satisfy
the condition $\left( E_{k}^{^{\prime \prime \prime }}4\right) $ for $x\in %
\left[ 0,\infty \right) -\left\{ 7\right\} .$ Indeed, we have%
\begin{eqnarray*}
d\left( 7,x\right) &\leq &h\max \left\{ 0,0,d\left( 7,x\right) ,d\left(
x,7\right) ,d\left( x,7\right) \right\} \\
&=&hd\left( 7,x\right) ,
\end{eqnarray*}%
a contradiction with $h\in \left( 0,1\right) $. So $E\left[ -2,0,2;21\right] 
$ is not a unique fixed $3$-ellipse of $T.$ For example, $E\left[ -1,0,1;21%
\right] $ is another fixed $3$-ellipse of $T.$
\end{example}

\begin{example}
\label{exmp7} Let $X=%
%TCIMACRO{\U{211d} }%
%BeginExpansion
\mathbb{R}
%EndExpansion
$ be the usual metric space. Let us take a $3$-ellipse $E\left[ -1,0,1;27%
\right] $ such as%
\begin{equation*}
E\left[ -2,0,2;27\right] =\left\{ x\in 
%TCIMACRO{\U{211d} }%
%BeginExpansion
\mathbb{R}
%EndExpansion
:\left\vert x+2\right\vert +\left\vert x\right\vert +\left\vert
x-2\right\vert =27\right\} =\left\{ -9,9\right\} .
\end{equation*}%
Let us define the self-mapping $T:%
%TCIMACRO{\U{211d} }%
%BeginExpansion
\mathbb{R}
%EndExpansion
\rightarrow 
%TCIMACRO{\U{211d} }%
%BeginExpansion
\mathbb{R}
%EndExpansion
$ as%
\begin{equation*}
Tx=\left\{ 
\begin{array}{ccc}
x & , & x\in \left\{ -9,9\right\} \\ 
\frac{1}{x+1} & , & x\in 
%TCIMACRO{\U{211d} }%
%BeginExpansion
\mathbb{R}
%EndExpansion
-\left\{ -9,9\right\}%
\end{array}%
,\right.
\end{equation*}%
for all $x\in 
%TCIMACRO{\U{211d} }%
%BeginExpansion
\mathbb{R}
%EndExpansion
.$ Then $T$ fixes the $3$-ellipse $E\left[ -1,0,1;27\right] $ but $T$ does
not satisfy the condition $\left( E_{k}^{^{\prime \prime \prime }}2\right) .$
\end{example}

In the following theorem, we investigate a contractive condition excludes
the identity map $I_{x}:X\rightarrow X$, defined by $I_{x}\left( x\right) =x$
for all $x\in X,$ in Theorems \ref{thrm1}, \ref{thrm2}, \ref{thrm3} and \ref%
{thrm4}.

\begin{theorem}
\label{thrm5} Let $\left( X,d\right) $ be a metric space, $E\left[
x_{1},...,x_{k};r\right] $ any $k$-ellipse on $X$ and the mapping $\xi
:X\rightarrow \left[ 0,\infty \right) $ defined as in Theorem \ref{thrm1}. $%
T $ satisfies the condition $\left( I_{k}\right) $%
\begin{equation*}
\left( I_{k}\right) \text{ \ \ \ \ \ \ }d\left( x,Tx\right) \leq \frac{\left[
\xi \left( x\right) -\xi \left( Tx\right) \right] }{k+1},
\end{equation*}%
for all $x\in X$ if and only if $T=I_{x}.$
\end{theorem}

\begin{proof}
Assume that $T$ satisfies the condition $\left( I_{k}\right) $ and $x\notin
Fix\left( T\right) $ for $x\in X.$ Then we get%
\begin{eqnarray*}
d\left( x,Tx\right) &\leq &\frac{\left[ \xi \left( x\right) -\xi \left(
Tx\right) \right] }{k+1} \\
&=&\frac{1}{k+1}\left[ \sum_{i=1}^{k}d\left( x,x_{i}\right)
-\sum_{i=1}^{k}d\left( Tx,x_{i}\right) \right] \\
&\leq &\frac{1}{k+1}\left[ kd\left( x,Tx\right) +\sum_{i=1}^{k}d\left(
Tx,x_{i}\right) -\sum_{i=1}^{k}d\left( Tx,x_{i}\right) \right] \\
&=&\frac{k}{k+1}d\left( x,Tx\right) <d\left( x,Tx\right) ,
\end{eqnarray*}%
a contradiction. Hence it should be $Tx=x$ for each $x\in X$ and $T=I_{x}.$
The converse statement is clearly proved.
\end{proof}

\begin{remark}
\label{rmk7} If a self-mapping $T:X\rightarrow X$ satisfying the conditions
of Theorem \ref{thrm1} (resp. Theorem \ref{thrm2}, Theorem \ref{thrm3} and
Theorem \ref{thrm4} ) does not satisfy the condition $\left( I_{k}\right) $
then we exclude the identity map.
\end{remark}

\section{\textbf{An Application to ($SReLU$)}}

In this section, we investigate a new application to the activation
functions using the notion of fixed $k$-ellipse. Why do we choose an
activation function? Activation functions are extensively used and important
part in neural networks since they decide whether a neuron should be
activated or not. In the literature, there are a lot of examples of
activation functions. For example, one of them is $S$-Shaped Rectified
Linear Activation Unit ($SReLU$) defined by%
\begin{equation*}
SReLU\left( x\right) =\left\{ 
\begin{array}{ccc}
t_{l}+a_{l}\left( x-t_{l}\right)  & , & x\leq t_{l} \\ 
x & , & t_{l}<x<t_{r} \\ 
t_{r}+a_{r}\left( x-t_{r}\right)  & , & x\geq t_{r}%
\end{array}%
\right. ,
\end{equation*}%
where $t_{l},$ $a_{l},$ $t_{r},$ $a_{r}$ are parameters \cite{Jin}.

Let us take $X=%
%TCIMACRO{\U{211d} }%
%BeginExpansion
\mathbb{R}
%EndExpansion
,$ $t_{l}=-6,$ $t_{r}=6,$ $a_{l}=2$ and $a_{r}=3.$ Then we get%
\begin{equation*}
SReLU\left( x\right) =\left\{ 
\begin{array}{ccc}
2x+6 & , & x\leq -6 \\ 
x & , & -6<x<6 \\ 
3x-12 & , & x\geq 6%
\end{array}%
\right.
\end{equation*}%
for all $x\in 
%TCIMACRO{\U{211d} }%
%BeginExpansion
\mathbb{R}
%EndExpansion
$ as seen in Figure \ref{fig:7}.

\begin{figure}[h]
\centering
\includegraphics[width=.6\linewidth]{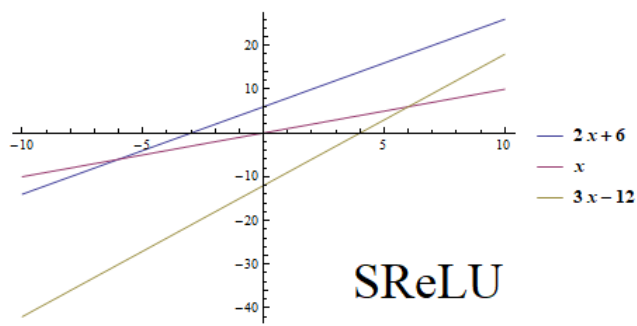} 
\caption{\Small The activation
function $SReLU\left( x\right) $.} \label{fig:7}
\end{figure}

Then we can easily say that the activation function $SReLU\left( x\right) $
fixes at least one $k$-ellipse on $%
%TCIMACRO{\U{211d} }%
%BeginExpansion
\mathbb{R}
%EndExpansion
.$ For example, let us consider the following $3$-ellipses:%
\begin{equation*}
E\left[ -1,0,1;15\right] =\left\{ x\in 
%TCIMACRO{\U{211d} }%
%BeginExpansion
\mathbb{R}
%EndExpansion
:\left\vert x+1\right\vert +\left\vert x\right\vert +\left\vert
x-1\right\vert =15\right\} =\left\{ -5,5\right\} ,
\end{equation*}%
\begin{equation*}
E\left[ -2,0,2;6\right] =\left\{ x\in 
%TCIMACRO{\U{211d} }%
%BeginExpansion
\mathbb{R}
%EndExpansion
:\left\vert x+2\right\vert +\left\vert x\right\vert +\left\vert
x-2\right\vert =6\right\} =\left\{ -2,2\right\} ,
\end{equation*}%
and%
\begin{equation*}
E\left[ -\alpha ,0,\alpha ;9\right] =\left\{ x\in 
%TCIMACRO{\U{211d} }%
%BeginExpansion
\mathbb{R}
%EndExpansion
:\left\vert x+\alpha \right\vert +\left\vert x\right\vert +\left\vert
x-\alpha \right\vert =9,\alpha \in 
%TCIMACRO{\U{211d} }%
%BeginExpansion
\mathbb{R}
%EndExpansion
\right\} =\left\{ -3,3\right\} .
\end{equation*}%
Then $E\left[ -1,0,1;15\right] ,$ $E\left[ -2,0,2;6\right] $ and $E\left[
-\alpha ,0,\alpha ;9\right] $ are fixed $3$-ellipse of $SReLU.$ Also, it can
be easily seen that the activation function $SReLU$ satisfies the conditions 
$\left( E_{k}1\right) $ and $\left( E_{k}2\right) $ (resp. $\left(
E_{k}^{^{\prime }}1\right) $ and $\left( E_{k}^{^{\prime }}2\right) ,$ $%
\left( E_{k}^{^{^{\prime \prime }}}2\right) $). Therefore, we say that there
exists at least one fixed $k$-ellipse of $SReLU.$ But $SReLU$ does not
satisfy the conditions $\left( E_{k}3\right) $ and $\left( E_{k}^{^{\prime
}}3\right) .$ Hence the fixed $k$-ellipse of $SReLU$ is not unique as seen
in the above examples. If we pay attention, this activation function $SReLU$
fixes infinite number of $k$-ellipses. This case is important in view of
increasing the number of fixed points in neural networks.

\end{document}